\documentclass{amsart}

\usepackage{amssymb}
\usepackage{mathrsfs}
\usepackage{amscd}
\usepackage[active]{srcltx}
\usepackage{verbatim}

\usepackage{newsymbol}
\let\mathcal \undefined
\def\mathcal{\mathscr}

\let\emptyset \undefined
\let\ge       \undefined
\let\le       \undefined
\newsymbol\le          1336  
\newsymbol\ge          133E  
\newsymbol\emptyset    203F
\newsymbol\notle       230A
\newsymbol\notge       230B

\theoremstyle{plain}
\newtheorem{theorem}{Theorem}
\theoremstyle{remark}
\newtheorem{remark}[theorem]{Remark}

\theoremstyle{plain}

\newtheorem{lemma}[theorem]{Lemma}

\def\R{{\mathbb R}}

\renewcommand{\P}{{\mathbb P}}

\renewcommand{\d}{\delta}
\newcommand{\eps}{\varepsilon}

\newcommand{\s}{^*}
\newcommand{\lb}{\langle}
\newcommand{\rb}{\rangle}

\newcommand{\n}{\Vert}
\newcommand{\one}{{{\bf 1}}}

\allowdisplaybreaks

\begin{document}

\title{Compactness in the Lebesgue-Bochner spaces $L^p(\mu;X)$}

\author{Jan van Neerven}
\address{Delft Institute of Applied Mathematics
\\Delft University of Technology
\\P.O. Box 5031, 2600 GA Delft
\\The Netherlands}
\email{J.M.A.M.vanNeerven@tudelft.nl}



\date\today

\dedicatory{Dedicated to the memory of Adriaan Cornelis Zaanen, on the occasion of his 100th birthday}

\thanks{The author was supported by
VICI subsidy 639.033.604 in the `Vernieuwingsimpuls' programme of the
Netherlands Organization for Scientific Research (NWO)}

\begin{abstract} Let $(\Omega,\mu)$ be a finite measure space, $X$ a Banach space,
and let $1\le p<\infty$. 
The aim of this paper is to give an elementary proof of the 
Diaz--Mayoral theorem that a subset $V\subseteq L^p(\mu;X)$ is relatively compact in $L^p(\mu;X)$ if 
and only if it is uniformly $p$-integrable, uniformly tight, and scalarly relatively compact.
\end{abstract}

\maketitle

Let $(\Omega,\mu)$ be  
a finite measure space and let $X$ be a Banach space. For $1\le p<\infty$ we denote by
$L^p(\mu;X)$ the Banach space of all strongly $\mu$-measurable functions $f: \Omega\to X$
for which $$\n f\n_p := \Big(\int_\Omega \n f\n^p\,d\mu\Big)^{1/p}$$ is finite.
Recall that a function $f:\Omega\to X$ is {\em strongly $\mu$-measurable} if it is the 
pointwise limit $\mu$-almost everywhere
of a sequence of simple functions $f_n:\Omega\to X$. 

Our aim is to present an elementary proof of the following characterisation of 
relative norm compactness in $L^p(\mu;X)$, due to Diaz and Mayoral \cite{DM}. 
Recall that a set $V\subseteq L^p(\mu;X)$ is called {\em uniformly $L^p$-integrable} if
$$\lim_{r\to\infty} \sup_{f\in V} \big\n \one_{\{\n f\n > r\}}f \big\n_p = 0,$$
{\em uniformly tight} if for all $\eps>0$ there exists a compact set $K\subseteq X$ such that
$$\sup_{f\in V} \mu(\{f\not\in K\}) \le \eps,$$ 
and {\em scalarly relatively compact} if for all $x\s\in X\s$ the set $\{\lb f,x\s\rb:\ f\in V\}$ 
is relatively compact in $L^p(\mu)$.

\begin{theorem}\label{thm:1} Let $1\le p<\infty$. 
A subset $V\subseteq L^p(\mu;X)$ is relatively compact if 
and only if it is uniformly $L^p$-integrable, uniformly tight, and scalarly relatively compact.
\end{theorem}

The proof in \cite{DM} relies on the Diestel--Ruess--Schachermayer characterisation \cite{DRS} 
of weak compactness in $L^1(\mu;X)$ and the notion of
Bocce oscillation, which was introduced and studied by Girardi \cite{G} and Balder--Girardi--Jalby \cite{BGJ} 
in the context of compactness in $L^1(\mu;X)$. An extension of the Diaz--Mayoral 
result to a class of $X$-valued Banach function
spaces, with an alternative proof based on Prohorov's tightness theorem, was obtained in \cite{N}.
A truly elementary proof of Theorem \ref{thm:1} seems to be lacking, and the present paper aims to fill this gap.

The main step in the `if' part of the proof is the following observation.

\begin{lemma}
If $V$ is a scalarly relatively 
compact subset in $L^p(\mu;X)$ 
consisting of functions taking values
in a fixed compact subset $K$ of $X$, then $V$ is relatively compact
 in $L^p(\mu;X)$. 
\end{lemma}
\begin{proof} 
We may assume that $\mu(\Omega) =1$.

It suffices to show that for all $\eps>0$, every sequence $(f_n)_{n\ge 1}$ in $V$ has a subsequence $(f_{n_k})_{k\ge 1}$
such that $\limsup_{k,l\to\infty} \n f_{n_k} - f_{n_l}\n_p \le \eps.$
Applying this inductively with a sequence $\eps_m\downarrow 0$, a diagonal argument will produce a 
convergent subsequence of  $(f_n)_{n\ge 1}$.

Since the compact set $K$ is contained in a separable
closed subspace of $X$, in the rest of the proof we may assume that $X$ is separable. 
Since every separable Banach space embeds  into $C[0,1]$ isometrically 
and the latter  has a Schauder basis (see, e.g.,
\cite[pages 35--36]{D}), 
we may assume that $X$ has a Schauder basis $(x_n)_{n\ge 1}$.

Let $X_N$ be the linear span of the basis vectors $x_1,\dots,x_N$.
For large enough $N$ we have we have $\sup_{x\in K} d(x,X_N)<\eps$.
Fixing such an $N$, choose a finite collection $y_1,\dots,y_M \in X_N$
such that for each $x\in K$ we have $d(x,y_m)<\eps$ for some $1\le m\le M$. For $n\ge 1$ and $1\le m\le M$  
let $\Omega_{mn}$ be the set of all $\omega\in\Omega$ such 
that $d(f_n(\omega),y_k)$ takes its first minimum at $k=m$, i.e.,
$\omega\in\Omega_{mn}$ if and only if 
$$ d(f_n(\omega),y_m) = \min_{1\le k\le M} d(f_n(\omega),y_k)$$
and $$d(f_n(\omega),y_k) > d(f_n(\omega),y_m) \ \ \hbox{ for } 1\le k\le m-1.$$  
The function
$\tilde f_n = \sum_{m=1}^M\one_{\Omega_{mn}}\otimes y_m$ is strongly measurable and satisfies
 $\n f_n - \tilde f_n\n_\infty< \eps$. 
Denoting by $P_N$ the partial sum projection onto $X_N$, it follows that
$$ \n f_n - P_Nf_n\n_p \le \n \tilde f_n - P_N\tilde f_n\n_p + \n (\tilde f_n- f_n) - P_N(\tilde f_n-f_n)\n_p
\le 0 + (1+\n P_N\n) \eps.$$
Clearly the functions $P_N f_n$ form a scalarly relatively compact set in $L^p(\Omega;X_N)$ and therefore, by arguing coordinatewise,
we find a subsequence $(P_Nf_{n_k})_{k\ge 1}$ and a function $g\in L^p(\Omega;X_N)$ such that $P_Nf_{n_k}\to g$
in $L^p(\Omega;X_N)$. 
It follows that 
$$ \limsup_{k,l\to\infty} \n f_{n_k} - f_{n_l}\n_p \le 2(1+\n P_N\n)\eps 
+ \limsup_{k,l\to\infty} \n P_Nf_{n_k} - P_Nf_{n_l}\n_p
= 2(1+\n P_N\n)\eps.
$$
As $\sup_{N\ge 1} \n P_N\n <\infty$ (see, e.g., \cite[pages 32--33]{D}), this concludes the proof. 
\end{proof}

\begin{proof}[Proof of Theorem \ref{thm:1}] 
We may assume that $\mu(\Omega) =1$.

`If': 
Suppose that $V$ is uniformly $p$-integrable, uniformly tight, and scalarly relatively compact.
As in the proof of the lemma, in order to prove that $V$ is relatively compact it suffices to show that 
 for all $\eps>0$, every sequence $(f_n)_{n\ge 1}$ in $V$ has a subsequence $(f_{n_k})_{k\ge 1}$
such that $\limsup_{k,l\to\infty} \n f_{n_k} - f_{n_l}\n_p \le \eps.$

Fix a sequence $(f_n)_{n\ge 1}$ in $V$ and a number $\eps>0$, and 
choose $r>0$ large enough such that $\sup_{n\ge 1} \n \one_{\{\n f_n\n > r\}}f_n \n_p < \eps$.
Choose a compact set $K\subseteq X$ such that $\sup_{n\ge 1} \mu\{f_n\not\in K\} < \eps^p/r^p.$
Then, for all $n,m\ge 1$,
\begin{align*}\big\n f_n - f_m\big\n_p
 & \le 2\eps + \big\n \one_{\{\n f_n\n \le r\}} f_n - \one_{\{\n f_m\n \le r\}} f_m\big\n_p  
\\ & \le 2\eps + \big\n \one_{\{\n f_n\n \le r, \ f_n\not\in K\}} f_n - \one_{\{\n f_m\n \le r,\ f_m\not\in K\}} f_m\big\n_p  
\\ & \quad\quad\ +\big\n \one_{\{\n f_n\n \le r, \ f_n\in K\}} f_n - \one_{\{\n f_m\n \le r, \ f_m\in K\}} f_m\big\n_p  
\\ & \le 2\eps + r\mu(\{f_n\not\in K\})^{1/p}  + r\mu(\{f_m\not\in K\})^{1/p} 
\\ & \quad\quad\ +\big\n \one_{\{f_n\in K\}} f_n - \one_{\{f_m\in K\}} f_m\big\n_p  
\\ & \quad\quad\ +\big\n \one_{\{\n f_n\n > r, \ f_n\in K\}} f_n - \one_{\{\n f_m\n > r, \ f_m\in K\}} f_m\big\n_p  
\\ & \le 6\eps + \big\n \one_{\{f_n\in K\}} f_n - \one_{\{f_m\in K\}} f_m\big\n_p.
\end{align*}
The set $\{\one_K f: \ f\in V\}$ is scalarly relatively
compact (as $\n \lb \one_K f - \one_K g,x\s\rb\n_p \le \n \lb f - g,x\s\rb\n_p$ for 
all $f,g\in V$ and $x\s\in X\s$), and hence relatively compact by the lemma.
Passing to a subsequence for which $(\one_{\{f_{n_k}\in K\}} f_{n_k})_{k\ge 1}$ is convergent in $L^p(\mu;X)$,
we find that for large enough $k,l$,
$$\big\n f_{n_k} - f_{n_l}\big\n_p \le 6\eps + \big\n \one_{\{f_{n_k}\in K\}} f_{n_k} - \one_{\{f_{n_l}\in K\}} f_{n_l}\big\n_p
\le 7\eps.$$

\smallskip
`Only if': If $V$ is relatively compact, then $V$ is scalarly relatively compact.
The proof that $V$ is uniformly $L^p$-integrable is routine; it is included for the
reader's convenience.
Given $\eps>0$ choose an $\eps$-net $f_1,\dots,f_N\in V$ and fix $r\ge 1$ so large that 
$\sup_{1\le n\le N} \n \one_{\{\n f_n\n>r\}}f_n\n_p<\eps$.
Pick an arbitrary $f\in V$ and choose an index $1\le n\le N$ such that $\n f_n-f\n_p <\eps$.
Then, by Chebyshev's inequality,
$$
\begin{aligned}
\ &  \big\n  \one_{\{\n f\n>2r\}}f\big\n_p
\\ & \qquad \le \eps + \big\n  \one_{\{\n f\n>2r\}}f_n\big\n_p
\\ & \qquad \le \eps + \big\n  \one_{\{\n f\n>2r, \ \n f_n-f\n \le r\}}f_n\big\n_p 
+ \big\n \one_{\{\n f_n-f\n > r\}}f_n\big\n_p
\\ & \qquad \le \eps + \big\n \one_{\{\n f_n\n>r\}}f_n\big\n_p 
+ \big\n \one_{\{\n f_n-f\n > r, \ \n f_n\n \le r\}}f_n\big\n_p
 + \big\n \one_{\{ \n f_n\n > r\}}f_n\big\n_p
\\ & \qquad \le 3\eps + r \mu(\{\n f_n-f\n > r\})^{1/p} 
\\ & \qquad \le 3\eps + \n f_n - f\n_p 
\\ & \qquad \le 4\eps. 
\end{aligned}
$$
The proof that $V$ is uniformly tight is also standard;
again it is included for the reader's convenience.

Fix $\eps>0$ and choose $n_0\ge 1$ so large that $2^{2-pn_0} < \eps$. 
For every $n\ge n_0$ we choose a finite $2^{-2n}$-net $f_{1,n},\dots f_{N_n,n}$ in $V$.
Using the standard fact (see, e.g., \cite{VTC}) that the distribution of a strongly measurable
function is tight, we may choose a compact set $K_n\subseteq X$ such that
$$ \mu(\{f_{j,n}\in K_n\}) \ge 1-2^{-pn}, \qquad 1\le j\le N_n.$$
Let
$ L_n := \{x\in X: \ d(x,K_n) < 2^{-n}\}.$
Fixing an arbitrary $f\in V$, choose an index $1\le j\le N_n$
such that $\n f-f_{j,n}\n_p < 2^{-2n}$. 
Then, by Chebyshev's inequality,
$$
\begin{aligned}
 \mu(\{f\not\in L_n\})
&  \le \mu(\{\n f-f_{j,n}\n\ge 2^{-n}\}) + \mu(\{f_{j,n}\not\in K_{n}\})
\\ &  \le 2^{pn} \n f-f_{j,n}\n_p^p + 2^{-pn}
\le 2^{-pn}+2^{-pn} = 2^{1-pn}.
\end{aligned}
$$

Let
$ K := \overline{\bigcap_{n\ge n_0} L_n}.$
If finitely many open balls $B(x_i,2^{-n})$ cover $K_n$, then the open balls
$B(x_i, 2^{1-n})$ cover $\overline{L_n}$.
Hence $K$ is totally bounded and therefore compact.
For all $f\in V$,
$$ \mu(\{f\not\in K\}) \le \sum_{n\ge n_0} \mu(\{f \not\in L_n\})
\le \sum_{m\ge m_0} 2^{1-pm} \le 2^{2-pm_0} < \eps.
$$
This proves that $V$ is uniformly tight.
\end{proof}

\begin{remark}
In the theorem and the lemma we cannot replace `(scalarly) relatively compact' by
`(scalarly) compact'. To see this let $\Omega = \{\omega_1,\omega_2\}$
a two-point set, $\mu(\{\omega_1\}) = \mu(\{\omega_2\})=1/2$,
and $V$ the collection of all non-zero functions $f: \Omega\to \R^2$
taking values in the closed unit ball of $\R^2$. 

For each $c = (c_1,c_2)\in \R^2$,
$\lb f, c\rb$ takes values in the interval $[-|c|,|c|]$. We claim that
for each $y\in [-|c|,|c|]$ we have $y = \lb f, c\rb = c_1 f(\omega_1) + c_2 f(\omega_2)$
for a suitable choice of $f\in V$. Indeed, we may suppose that $c\not=0$. 
If $y\not = 0$ take $f(\omega_i) = \frac{c_i}{|c|^2} y$, $i=1,2$; if 
$y=0$ take $f(\omega_1) = \frac{c_2}{|c|}$ and $f(\omega_2) = -\frac{c_1}{|c|}$.
Hence, $V$ is scalarly compact. Clearly, $V$ is relatively compact in $L^2(\mu;\R^2)$, 
but not compact (as $0\not\in V$).
\end{remark}

Let $E$ be a Banach function space over the finite measure space $(\Omega,\mu)$ (see, e.g., \cite{Z}),
and define $E(X)$ as the Banach space of all strongly $\mu$-measurable functions 
$f:\Omega\to X$ such that $\omega\mapsto \n f(\omega)\n$ belongs to $E$.  
It has been shown in \cite{N} that if $E$ has order continuous norm, 
then Theorem \ref{thm:1} extends to $E(X)$ provided one replaces `uniformly $p$-integrable'
by `almost order bounded'; the proof is based on Prokhorov's compactness theorem.
Recall that a subset $F$  of a Banach lattice 
$E$ is called {\em almost order bounded} if for every $\eps>0$ there exists an element $x_\eps\in E_+$ such that
$$F\subseteq [-x_\eps,x_\eps]+ B(\eps),$$ where 
$[-x_\eps,x_\eps] := \{y\in E: \ -x_\eps\le y\le x_\eps\}$ and $B(\eps) :=\{x\in X: \ \n
x\n < \eps\}$. 
It was also shown in \cite{N}  that if $E$ has order continuous norm, 
then every almost order bounded set is {\em uniformly $E$-integable}, i.e.,
$$\lim_{r\to\infty} \sup_{f\in V} \big\n \one_{\{\n f\n > r\}}f \big\n_{E} = 0.$$

The question as to whether Theorem \ref{thm:1} extends to the Banach function space setting in its present
form was left open in \cite{N}. The answer is affirmative:

\begin{theorem}\label{thm:2} Let $E$ have order continuous norm. 
A subset $V\subseteq E(X)$ is relatively compact if 
and only if it is uniformly $E$-integrable, uniformly tight, and scalarly relatively compact.
\end{theorem}
\begin{proof}
By the above, the `only if' part follows from the results of \cite{N}.
The `if' part follows by extending the `if' part of the proof of Theorem \ref{thm:1} {\em mutatis mutandis}.
\end{proof}

The proof of the `only if' part of Theorem \ref{thm:1} does not immediately generalise: both the proof of uniform
integrability and tightness depend on Chebyshev's inequality. We finish the paper by indicating alternative and
elementary proof of the `only if' part 
that works in the setting of an order continuous Banach function space.

As was already pointed out, relatively compact sets in $E(X)$ are almost order bounded,
and if $E$ has order continuous norm, then every almost order bounded set in $E(X)$
is uniformly $E$-integrable for all $1\le p<\infty$. 
It thus remains to prove that if $V\subseteq E(X)$ is relatively compact and $E$ has order continuous norm, 
then $V$ is uniformly tight.
Arguing by contradiction, suppose $V$ fails to be uniformly tight. 

{\em Claim:} There exist $\eps_0>0$ and $\delta_0>0$
such that for all compact $K\subseteq X$ we can find $f\in V$ with $\mu(\{f\not\in K + B_{\d_0}\}) > \eps_0$.
Here, $B_{\delta_0}$ is the open ball centred at $0$ with radius $\delta_0$ and 
$K+B_{\delta_0} = \{x+y: \ x\in K, \, y\in B_{\delta_0}\}$.

If this were false, we could fix an arbitrary $\eps>0$ and select, for each $n\ge 1$, a compact set $K_n\subseteq X$ 
such that $\mu(\{f\not \in K_n + B_{1/2^n}\}) \le \eps/2^n$ for all $f\in V$. 
The set $K:= \bigcap_{n\ge 1}\overline{K_n + B_{1/2^n}}
$ is totally bounded and therefore compact, and $\mu(\{f\not\in K\}) \le\sum_{n\ge 1} \eps/2^n = \eps$ for all $K$.
This contradicts the assumption that $V$ fails to be uniformly tight. This proves the claim.

Fix an arbitrary $f_1\in V$ and choose a
compact set $K_1\subseteq X$ such that $\mu(\{f_1\not\in K_1\}) \le \frac12\eps_0$.
Choose $f_2\in V$ such that $\mu(\{f_2\not\in K_1 + B_{\delta_0}\}) > \eps_0$ and choose a compact 
$K_2\subseteq X\setminus B_0$ such that $K_1\subseteq K_2$ and $\mu(\{f_2\not\in K_2\}) \le \frac12\eps_0$.
Proceed inductively we arrive at a sequence $(f_n)_{n\ge 1}\subseteq V$ and a, increasing sequence 
$(K_n)_{n\ge 1}$ of compact sets in $X\setminus B_0$ with the following properties:
\begin{enumerate}
 \item [\rm (i)] $\mu(\{f_j\not\in K_m\}) \le \frac12 \eps_0$ whenever $1\le j\le m$;
 \item [\rm(ii)] $\mu(\{f_n\not\in K_m + B_{\delta_0}\}) > \eps_0$ whenever $n > m$.
\end{enumerate}
For $1\le m < n$ set 
$$\Omega_{m,n} := \{f_m \in K_m, \ f_n\not\in K_m + B_{\delta_0}\}.$$
Then $\mu(\Omega_{m,n}) \ge \frac12\eps_0$, and
on $\Omega_{m,n}$ we have the pointwise inequality $\n f_n - f_n\n_X \ge \delta_0$. Hence,
$$ \n f_n - f_m \n_{E(X)}  \ge \delta_0 \n \one_{\Omega_{m,n}}\n_{E}.
$$
 As a consequence of the next lemma, the sequence $(f_n)_{n\ge 1}$ has no convergent subsequence.
This completes the proof.

\begin{lemma} In the above setting, for all $r>0$
we have $\inf \{\n \one_A \n_{E}: \ \mu(A) \ge r\} >0.$
\end{lemma}
\begin{proof}
If this were false, we could find $r_0>0$ and a sequence of sets $A_n$ satisfying $\mu(A_n)\ge r_0$
such that $\n \one_{A_n} \n_{E_\mu} \le 2^{-n}$.
Then, 
$\n \sum_{n\ge 1} \one_{A_n}\n_{E} \le 1.$
But this contradicts the fact that $\mu(A_n) \ge r$, since the latter implies 
$\sum_{n\ge 1} \one_{A_n}  = \infty$ on some set of $\mu$-measure $\frac12 r$. 
\end{proof}


%
%

\end{document}